\documentclass[letterpaper]{article}

\usepackage[english]{babel}
\usepackage[utf8x]{inputenc}
\usepackage[T1]{fontenc}
\usepackage{fullpage}

\usepackage{amsmath}
\usepackage{graphicx}
\usepackage{amsthm}
\newtheorem{theorem}{Theorem}
\newtheorem*{proposition}{Proposition}
\newtheorem*{corollary}{Corollary}

\usepackage{amssymb} 
\theoremstyle{definition}

\newtheorem*{remark}{Remark}

\title{River Crossing Problems: Algebraic Approach}
\date{\vspace{-5ex}}

\begin{document}
\maketitle

\begin{center}
\author{ Elena A. Efimova$^1$}
\end{center}

\footnotetext[1]{Intelligent Systems Department, Russian State University for the Humanities; {\tt yefi03@yandex.ru}.}

\begin{abstract}
We consider two river crossing problems, about jealous husbands and about missionaries and cannibals. The missionaries and cannibals problem arose a thousand years after the jealous husbands problem, although its solution had actually appeared several hundred years before its formulation. We apply an algebraic approach to study these problems, using a symmetry group action on the state set of the jealous husband problem; then category theory is used to describe the relationship between the two problems. Some historical issues are also touched, related to the fact that the missionaries and cannibals problem arose precisely when the group approach began to be widely spread and popularized. This is the approach that naturally connects both problems.
\end{abstract}


\section{Introduction} 
A river crossing problem is that objects are to cross the river from one bank to another using a boat of limited capacity. 
The paper is devoted to the jealous husbands problem  and the missionaries and cannibals problem. The first problem is the following. 
\textit{Three couples must cross a river using a boat that holds at most two people. No husband wants his wife to be anywhere, ashore or in the boat, with other men without his presence. How can they cross the river?} 
The second problem is similar, but instead of three couples, it is for three missionaries and three cannibals with the following constraint: nowhere the number of cannibals must be greater than the number of missionaries.

There are variants of these problems for four or five couples and, respectively, for four or five missionaries and the same number of cannibals where the boat's capacity is three people. In the paper the general case is considered.

The jealous husbands problem first appeared in Alcuin's collection \textit{Propositiones ad Acuendos Juvenes} (en. "Problems to Sharpen the Young") about 800 AD (see \cite{hadley}). The missionaries and cannibals problem was formulated in the late 1870s (see \cite{singmaster}). Since the early 1960s, these puzzles have been used as illustrations of artificial intelligence methods, where the state space graph is usually constructed to find a solution (see Schwarz in \cite{schwarz}, Amarel in \cite{amarel}). In \cite{schwarz} Schwarz uses adjacency matrix of the graph. Bellman applies dynamic programming in \cite{bellman}. In the paper \cite{fraley}, Fraley, et al. place the graph on the coordinate plane; sometimes the states are located at the vertices of the hypercube. 

We use algebraic methods to find solutions and study the relationship between these problems. First the state spaces are considered. Then the symmetric group action on the state set of the jealous husbands problem is introduced, and we show how by using a solution of one problem another problem can be solved. We also deal with the question of the capacity of the boat. The relationship between puzzles is described using category theory. In conclusion, we turn to historical issues related to the time of the occurrence of the missionaries and cannibals problem.

{\noindent}{\bf Acknowledgement.} I am grateful to Alexander I. Efimov for useful discussions.

\section{Space of states}
In this section, we consider the space of admissible states, which is modelled as a graph, where states are vertices and  transitions between them are edges. We use the abbreviation \textit{HW} for the jealous husbands problem and \textit{MC} for the missionaries and cannibals problem.  

Denote by $n$ the number of couples in the \textit{HW} problem, and, accordingly, we have $n$ missionaries and $n$ cannibals in the \textit{MC} problem, where $n>1$. Also denote by $b$ a capacity of the boat, that is, the least number of people that it can carry such that the problem has a solution. It is known that
\[ 
b = 
\begin{cases}
       2 \text{ if } n \leq 3;\\
       3 \text{ if } n = 4 \text{ or } 5;\\
       4 \text{ if } n \geq 6
\end{cases}
\]
(see below for details). In particular, $b=n$ if $n=2$, otherwise $b < n$.

The subset of people is safe in the \textit{HW} problem  if either it does not include any husband or it contains the husband of each wife belonging to it. Similarly, the subset of people is safe in the \textit{MC} problem if either no missionary belongs to it or the number of cannibals is not greater than the number of missionaries.

The state of the problem, \textit{HW} or \textit{MC}, is determined by a subset of people who are on the left bank or on the right bank, whereas no one is in the boat, and also by the boat's location. The state is admissible if both subsets of people, on the left and on the right, are safe. Only admissible states and admissible transitions between them will be considered below.

It is clear that the only way that husbands can be divided into two non-empty subsets, corresponding to the left bank and to the right bank of the river, is that each husband is with his wife. Likewise, missionaries can be separated into two non-empty groups only if the number of missionaries coincides with the number of cannibals, both on the left and on the right.

Therefore, in both problems, there are three types of states, determined by where the husbands or the missionaries are: (1) only on the left side, (2) only on the right side, or (3) on both sides.   

Consider the \textit{HW} problem. 
It is assumed that each wife is assigned a unique number from 1 to $n$ and that husbands are assigned the numbers of their wives. Denote by $h_i$ and $w_i$ the husband and the wife with $i$ number, respectively, for $i = 1, 2, \dots, n$. Let $W$ be the set of wives, $H$ be the set of husbands, and $N$ be the set of numbers, so that $N = \{1, 2, \dots, n\}$. Also, we put $U = W \cup H$.

By $(V, E)$ denote the state space of the problem, where $V$ is a set of states and $E$ is a set of transitions. 
The state of the \textit{HW} problem will be described as a triple $(L, R, loc)$, where $L$ is a subset of wives and husbands on the left bank, $R$ is a subset of people on the right bank, and $loc$ is a location of the boat such that $loc$ is equal to $left$ or $right$ if the boat is on the left or on the right, accordingly. 
It is obvious that $R = U \setminus L$, so that the state is uniquely determined by the items $L$ (or $R$) and $loc$, but we will use the triple for the sake of further constructions. 

So, there are three types of states:

\begin{enumerate}
\item[(a)] $\left( \bigl\{w_{i_1},\dots,w_{i_p}\bigr\}, \bigl\{w_{i_{p+1}},\dots,w_{i_n},h_1,\dots,h_n \bigr\},loc\right)$, $0 \le p \le n$;
\item[(b)] $\left(\bigl\{w_{i_1},\dots,w_{i_p},h_1,\dots,h_n\bigr\},\bigl\{w_{i_{p+1} },\dots,w_{i_n} \bigr\},loc\right)$, $0 \leq p \leq n$; 
\item[(c)] $\left(\bigl\{w_{i_1},\dots,w_{i_p},h_{i_1},\dots,h_{i_p}\bigr\},\bigl\{w_{i_{p+1}},\dots,w_{i_n},h_{i_{p+1}},\dots,h_{i_n} \bigr\},loc\right)$, $0 < p < n$.
\end{enumerate}

The transition $state_1 \stackrel{f}\longrightarrow state_2$ is a pair of states $state_1$ and $state_2$, such that $state_2$ can be reached from $state_1$ by using one trip on the boat, with which a move $f = (B, loc)$ is associated, where $B$ is a safe set of people in the boat, $0 < | B | \leq b$, and $loc$ denotes the river bank from which the boat sails. Suppose $state_1 = (L_1, R_1, loc_1)$ and $state_2 = (L_2, R_2, loc_2)$. Then $loc_1 \neq loc_2$, $loc = loc_1$, $B = L_1 \setminus L_2$ if $loc_1=left$, and $B = R_1 \setminus R_2$ if $loc_1=right$.  

Consider the transition $state_1 \stackrel{f}\longrightarrow state_2$, where $f = (B, left)$.

First suppose that the state $state_1$ is of the type (a). Then the set $B$ can consist only of the following set of people:
\begin{enumerate}
\item[i)] A subset of wives from the left bank.
\end{enumerate}

Next, if $state_1$ is of the type (b), then there are three variants for the set $B$:

\begin{enumerate}
\item[ii)] A subset of wives from the left bank;
\item[iii)] All husbands if $n = b = 2$;
\item[iv)] Husbands of all wives from the right bank of the river and maybe some couples from the left bank.
\end{enumerate}

Finally, let $state_1$ be of the type (c). Then there are two cases for the set $B$:
\begin{enumerate}
\item[v)] Some wives and all husbands from the left bank;
\item[vi)] Some couples from the left bank. 
\end{enumerate}

In all cases, it is evident that $state_2$ is one of the above types (a), (b), or (c). 
Transportation from the right bank to the left bank is similar.

It should be noted that the number of wives both on each bank and in the boat is not greater than the number of husbands if the latter is not zero.

Now let us turn to the \textit{MC} problem. The state space of this problem is a pair $(V', E')$, where $V'$ is a state set and $E'$ is a set of transitions between states.
The state of the \textit{MC} problem is described as a triple $(L', R', loc)$ and the move is a pair $(B', loc)$, where the items $L'$, $R'$, and $B'$ are pairs of the form $(c, m)$ such that $c$ is a number of cannibals and $m$ is a number of missionaries, respectively, on the left bank, on the right bank, and in the boat. The $loc$ item, as before, indicates a location of the boat. 

From the above it follows that there are three types of states in the \textit{MC} problem:

\begin{enumerate}
\item[(a')] $((x, 0), (y, n), loc)$, $0 \leq x \leq n$;
\item[(b')] $((x, n), (y, 0), loc)$, $0 \leq x \leq n$;
\item[(c')] $((x, x), (y, y), loc)$, $0 < x < n$,
\end{enumerate}
where $y = n - x$.

Consider the transition $state'_1 \stackrel{f'}\longrightarrow state'_2$ from the left bank to the right bank.

First let $state'_1$ be of the type (a'). We have: 
\begin{enumerate}
\item[i')] $f'=((c, 0), left)$, $state'_2=((x - c, 0), (y + c, n), right)$, where $0 < c \leq min(x, b)$.   
\end{enumerate}

Next suppose $state'_1$ is of the kind (b') and $f' = ((c, m), left)$; then it is easy to see that $m = 0$ or $m = n$ or $m = y + c$. Therefore, there are three cases:
\begin{enumerate}
\item[ii')] $f' = ((c, 0), left)$, $state'_2 = ((x - c, n), (y + c, 0), right)$, where $0 < c \leq min(x, b)$;
\item[iii')] $f' = ((0, n), left)$ if $n = b = 2$, $state'_2 = ((x, 0), (y, n), right)$;
\item[iv')] $f' = ((c, y + c), left)$, $state'_2 = ((x - c, x - c), (y + c, y + c), right)$, where $0 \leq c \leq x$ and $0 < y + 2c \leq b$.
\end{enumerate}

Now let $state'_1$ be of the kind (c'). Then we obtain two cases:
\begin{enumerate}
\item[v')] $f' = ((c, x), left)$, $state'_2 = ((x - c, 0), (y + c, n), right)$, where $0 \leq c \leq x$ and $0 < c + x \leq b$;
\item[vi')] $f' = ((c, c), left)$, $state'_2 = ((x - c, x - c), (y + c, y + c), right)$, where $0 < c \leq x$ and $2c \leq b$.
\end{enumerate}

Transportation from the right bank to the left bank is similar.

Having a description of all admissible states and transitions, one can find all solutions in both problems.
Besides, one can observe that there is a correspondence between states and transitions of these problems. If a solution of the \textit{HW} problem is given, then replacing subsets of wives and husbands with pairs of their amounts, we obtain a solution to the \textit{MC} problem. This is discussed more precisely below.

\section{Symmetric group action on states}
Let $S_n$ be a symmetric group, that is, a permutation group on the set of $n$ elements $N = \{1, 2,\dots, n\}$. Consider an $S_n$--group action on the set of states $V$ of the \textit{HW} problem.

Let $X$ be a subset of husbands and wives and $\pi \in S_n$. We put 
$$\pi X = \{w_{\pi(i)} ~|~  w_i \in X\} \cup \{h_{\pi(j)}  ~|~  h_j \in X\}.$$

\begin{remark}
If the set $X$ is safe, then the set $\pi X$ is also safe. Indeed, let the woman $w_j$ and the man $h_i$, for $i \neq j$, belong to the set $\pi X$. Then the woman $w_{\pi^{-1}(j)}$ and the man $h_{\pi^{-1}(i)}$ belong to the set $X$. Since the set $X$ is safe, the man $h_{\pi^{-1}(j)}$ also belongs to this set. Hence, the man $h_j$ belongs to the set $\pi X$.
\end{remark}

Now let us define an action of the group $S_n$ on the set $V$. If $state = (L, R, loc)$ and $\pi \in S_n$, then we put $\pi state = (\pi L, \pi R, loc)$. It is clear that $\pi R = U \setminus \pi L$.
The correspondence $state \mapsto \pi state$ for $\pi \in S_n$ defines an automorphism $\sigma_{\pi}\colon V \to V$. 
The homomorphism of the group $S_n$ to the group of automorphisms on the set $V$ such that $\pi \mapsto \sigma_{\pi}$ determines an $S_n$--group action on the set $V$, as we have	
\begin{itemize}
\item $ (\pi_1\pi_2)state = \pi_1(\pi_2 state)$ for any $state$ and $\pi_1, \pi_2  \in S_n$;
\item $e \, state = state$ for any $state \in V$, where $e$ denotes the identity permutation.
\end{itemize}

Denote by $S_nstate$ the orbit of $state$: $S_nstate = \{\pi state ~|~\pi \in S_n\}.$
The action of the group $S_n$, as usual, defines on the set $V$ the following equivalence relation:  
$state_1 \sim state_2 \text{  iff  } S_nstate_1 = S_nstate_2.$

After applying the permutation $\pi$ to the state $(L, R, loc)$, the subset of wives is ordered in ascending numbers both in the set $\pi L$ and in the set $\pi R$. The order of husbands' numbers does not matter. This will be important in Theorem 2.

An action of $S_n$-group on the set of moves is determined in the same way: $\pi(B, loc) = (\pi B, loc)$.

Let $state_1 \stackrel{f}\longrightarrow state_2$ be a transition in the \textit{HW} problem, and $\pi \in S_n$. Then  $\pi state_1 \stackrel{\pi f}\longrightarrow \pi state_2$ is also a transition in the \textit{HW} problem. Indeed, safety of $\pi state_1$, $\pi state_2$, and $\pi f$ follows from the remark above. Suppose that $state_1 = (L_1, R_1, loc_1)$, $state_2 = (L_2, R_2, loc_2)$, and $f = (B, loc_1)$.  Then $\pi B = \pi L_1 \setminus \pi L_2$ if $loc_1=left$ and $\pi B = \pi R_1 \setminus \pi R_2$ if $loc_1=right$.

The solution of the \textit{HW} problem is a sequence $Sol$ of states and transitions between them, which connects the initial state with the finite state, of the form
$$state_0 \stackrel{f_1}\longrightarrow state_1 \stackrel{f_2}\longrightarrow \dots \stackrel{f_k}\longrightarrow state_k,$$
where $f_i$ denotes a move from $state_{i - 1}$ to $state_i$, for $i = 1, 2, \dots, k$ (sometimes the moves will be omitted, as they are uniquely restored).

From the above it follows that the permutation $\pi$ takes the solution $Sol$ to the solution $\pi Sol$ of the form
$$\pi state_0 \stackrel{\pi f_1}\longrightarrow \pi state_1 \stackrel{\pi f_2}\longrightarrow \dots \stackrel{\pi f_k}\longrightarrow \pi state_k,$$
which differs from $Sol$ only by the numbering of people.
Obviously, $\pi state_0 = state_0$ and  $\pi state_k = state_k$.

Besides, notice that if $\pi$ belongs to the stabilizer of $state_{i - 1}$, then the sequence
\[
state_0 \stackrel{f_1}\longrightarrow \dots \stackrel{f_{i-1}}\longrightarrow state_{i-1} \stackrel{\pi f_i}\longrightarrow \pi state_i \dots \stackrel{\pi f_k}\longrightarrow \pi state_k
\]
is also a solution of the \textit{HW} problem, for $i = 1, 2, \dots, k$.

Now let $X$ be a subset of husbands and wives. Denote by $| X |_H$ and by $| X |_W$ the number of husbands and the number of wives, respectively, which are contained in the set $X$. 

Suppose that $(L_1, R_1, loc_1) \sim (L_2, R_2, loc_2)$. It is clear that $loc_1 = loc_2$, $| L_1 |_H = | L_2 |_H$, $| L_1 |_W = | L_2 |_W$, $| R_1 |_H = | R_2 |_H$, and $| R_1 |_W = | R_2 |_W$.

Denote by $V/\!\sim$ a set of state orbits, that is, a quotient set with respect to the equivalence relation $\sim$.

\begin{theorem}
There exists a one-to-one correspondence between the set of state orbits  of the \textit{HW} problem and the state set of the \textit{MC} problem. 
\end{theorem} 
\begin{proof}
Suppose that $state = (L, R, loc)$ is admissible in the \textit{HW} problem. Define the map $g \colon V/\!\sim \ \to V'$ by the rule $S_nstate \mapsto state'$, where  $state' = \bigl((| L |_W, | L |_H), (| R |_W, | R |_H), loc\bigr)$. As it was noted above, $state'$ is admissible in the \textit{MC} problem. 
From the remark above it follows that the mapping is well defined. 

Also, define the map $h \colon V' \to V /\! \sim $. Let $state' = ((c, m), (n - c, n - m), loc)$. We put $h(state') = S_nstate$, where $state = (L, R, loc)$, $L = \{w_1, w_2, \dots, w_c, h_1, h_2, \dots, h_m\}$, and $R = U \setminus L$. It is clear that if $state'$ is admissible, then $state$ is admissible, as well.

For these two mappings, we have
\begin{eqnarray*}
(h \circ g)\left(S_nstate\right) = h\left(g(S_nstate)\right) = h(state') = S_nstate;\\
(g \circ h)(state') = g(h(state')) = g\left(S_nstate\right) = state'.
\end{eqnarray*}

Hence, the mappings $g$ and $h$ are mutually inverse and one-to-one. The theorem is proved.
\end{proof}

\begin{corollary}
If the sequence
$$state_0 \rightarrow state_1 \rightarrow  \dots \rightarrow  state_k$$
is a solution of the \textit{HW} problem, then the sequence 
$$g\left(S_nstate_0\right) \rightarrow g\left(S_nstate_1\right) \rightarrow  \dots \rightarrow g\left(S_nstate_k\right)$$
is a solution of the \textit{MC} problem. 
\end{corollary}

So, in order to obtain a solution to the \textit{MC} problem from the solution of the \textit{HW} problem, it is obviously enough to omit the numbers of people and replace wives with cannibals and husbands with missionaries.

Further, let us show that having the solution of the \textit{MC} problem, one can construct, in turn, a certain subset of solutions of the \textit{HW} problem.

\begin{theorem}
Suppose the sequence 
\[ state'_0 \rightarrow state'_1 \rightarrow \dots \rightarrow state'_k\]
is a solution of the \textit{MC} problem. Then there exists a solution
\[state_0 \rightarrow state_1 \rightarrow  \dots \rightarrow  state_k\]
of the \textit{HW} problem such that, for each $j = 0, 1, \dots, k$, the following condition holds:
\[ 
state_j \in h\left(state'_j\right).\]    
\end{theorem}

\begin{proof}
For the initial states, we have $h(state'_0)=\{state_0\}$, so $state_0 \in h(state'_0)$. Let us show that for each $i = 1, 2, \dots, k$ there exists a transition $state_{i-1} \stackrel{f_i} \longrightarrow state_i$ such that $state_i \in h\left(state'_i\right)$. The proof is by induction on $i$. 

Consider the first transition. Since the initial state of the \textit{MC} problem  $\bigl((n, n), (0, 0), left\bigr)$ corresponds the case (b') (see Section 2), we have the following cases:
\begin{enumerate}
\item[ii')] $f'_1 = \bigl((c, 0), left\bigr)$, $state'_1 = \bigl((n - c, n), (c, 0), right\bigr)$, where $0 < c \le b$;
\item[iii')] $f'_1 = \bigl((0, n), left\bigr)$ if $n = b = 2$, $state'_1 = \bigl((n, 0), (0, n), right\bigr)$;
\item[iv')] $f'_1 = \bigl((c, c), left\bigr)$, $state'_1 = \bigl((n - c, n - c), (c, c), right\bigr)$, where $0 < 2c \le b$.
\end{enumerate}

Let us introduce some auxiliary notation. We put for $0 \le p \le q \le n$:
\[W_{p, q} = \bigl\{w_i \bigm | i = p, p + 1, \dots, q\bigr\} \text{ and } H_{p,  q} =\bigl\{h_j \bigm | j = p, p + 1, \dots, q\bigr\}.\]

The initial state of the \textit{HW} problem $\bigl(U, \emptyset, left\bigr)$ corresponds to the case (b). So, under the conditions indicated above, we put for the first transition:

\begin{enumerate}
\item[ii)] $f_1 = (W_{n - c + 1, n}, left)$, 
    $state_1 = (W_{1, n - c} \cup H, W_{n - c + 1, n}, right)$;
\item[iii)] $f_1 = (H, left)$, $state_1 = (W, H, right)$;
\item[iv)] $f_1 = (W_{n - c + 1, n} \cup H_{n - c + 1, n}, left)$, $state_1 = (W_{1, n - c} \cup H_{1, n - c}, 
 W_{n - c + 1, n} \cup H_{n - c + 1, n}, right)$.
\end{enumerate}

In all cases, $state_1 \in h(state'_1)$.

Suppose that the assertion holds for states from 1 to $i - 1$ and the transitions between them. Consider the next transition.
Let $state_{i-1}$ have the form 
\[
\bigl(\bigl\{w_{i_1},\dots,w_{i_p},h_{j_1},\dots,h_{j_q} \bigr\}, \bigl\{w_{i_{p+1}}, \dots, w_{i_n},h_{j_{q+1}},\dots,h_{j_n} \bigr\}, left\bigr),
\]
where $i_1 < \dots < i_p$ and $i_{p + 1} < \dots < i_n$, and $\pi$ be a permutation inverse to the following permutation:
\[
\begin{pmatrix}
1 & \cdots & p & p+1 & \cdots & n \\
i_1 & \cdots &i_p & i_{p+1} &\cdots &i_n 
\end{pmatrix}. 
\]

Applying $\pi$ to the previously constructed states from $state_0$ to $state_{i - 1}$ and to the transitions between them, we obtain a new path such that $\pi state_j \in h(state'_j)$,
for $j = 1, \dots, i - 1$, where all transitions are admissible. But $\pi state_{i-1}$ is equal to
\[
\Bigl(
\bigl\{w_1,\dots,w_p,h_{\pi(j_1)},\dots,h_{\pi(j_q)} \bigr\},  \bigl\{w_{p+1}, \dots, w_n,h_{\pi(j_{q+1})},\dots,h_{\pi(j_n)} \bigr\}, left\Bigr).
\]

So without loss of generality we can assume that $state_{i-1}$ has the above form. Under the conditions detailed in Section 2 for $state'_{i-1}$, we put for the transitions:

\begin{enumerate}
\item[(a)] $state_{i-1}=\bigl(W_{1,x}, W_{x+1,n} \cup H, left \bigr);$
\begin{enumerate}
\item[i)] $f_i = \bigl(W_{x - c + 1, x}, left\bigr),\  
state_i = \bigl(W_{1, x - c}, W_{x - c + 1, n} \cup H, right\bigr)$;
\end{enumerate}

\item[(b)] $state_{i-1}=\bigl(W_{1,x} \cup H, W_{x+1,n} , left \bigr);$

\begin{enumerate}
\item[ii)] $f_i = \bigl(W_{x - c + 1, x}, left\bigr),\  
state_i = \bigl(W_{1, x - c} \cup H, W_{x - c + 1, n}, right\bigr)$;
\item[iii)] $f_i = \bigl(H, left\bigr),\  
state_i = \bigl(W_{1, x}, W_{x + 1, n} \cup H, right\bigr)$;
\item[iv)] $f_i = \bigl(W_{x - c + 1, x} \cup H_{x - c + 1, n}, left\bigr), 
state_i = \bigl(W_{1, x - c} \cup H_{1, x - c}, W_{x - c + 1, n} \cup H_{x - c + 1, n}, right\bigr)$;
\end{enumerate}

\item[(c)] $state_{i-1}=\bigl(W_{1,x} \cup H_{1,x}, W_{x+1,n} \cup H_{x+1,n}, left \bigr);$

\begin{enumerate}
\item[v)] $f_i = \bigl(W_{x - c + 1, x} \cup H_{1, x}, left\bigr),  
state_i = \bigl(W_{1, x - c}, W_{x - c + 1, n} \cup H, right\bigr)$;
\item[vi)] $f_i = \bigl(W_{x - c + 1, x} \cup H_{x - c + 1, x}, left\bigr), 
state_i = \bigl(W_{1, x - c} \cup H_{1, x - c}, W_{x - c + 1, n} \cup H_{x - c + 1, n}, right\bigr)$.
\end{enumerate}
\end{enumerate}

Transportation from the right bank to the left bank is considered similarly. 

Therefore, $state_i \in h\{state'_i\}$.
At step $k$, since $h(state'_k) = \{state_k\}$, the transition is carried out to the final state, so that we get a solution of the \textit{HW} problem. Theorem 2 is proved. 
\end{proof}

By constructing the orbits of states and indicating transitions between them, we can obtain the whole set of solutions of the \textit{HW} problem, which correspond to the solution of the \textit{MC} problem. 

For example, let us take the following solution of the \textit{MC} problem:
\begin{eqnarray*}
\begin{split}
&((3, 3), (0, 0), left) \to ((1, 3), (2, 0), right) \to ((2, 3), (1, 0), left) 	\to \\
&((0, 3), (3, 0), right) \to ((1, 3), (2, 0), left) \to ((1, 1), (2, 2), right) \to \\
&((2, 2), (1, 1), left) \to	((2, 0), (1, 3), right) \to ((3, 0), (0, 3), left) \to \\
&((1, 0), (2, 3), right) \to ((2, 0), (1, 3), left) \to	((0, 0), (3, 3), right).
\end{split}
\end{eqnarray*}

Let us construct a solution of the \textit{HW} problem, as shown in Theorem 2. The first transition is
\[
(U, \emptyset, left) \xrightarrow[]{(\{w_2,w_3 \},left)} (\{w_1, h_1, h_2, h_3\}, \{w_2, w_3\}, right).
\]

According to the theorem, we apply the permutation 
\[
\pi = \begin{pmatrix}
1&2&3\\3&1&2
\end{pmatrix},
\]
which is inverse to the permutation
\[
\begin{pmatrix}
1&2&3\\2&3&1
\end{pmatrix},
\]
to this path, and then perform the second transition. As a result, we get:
\[
(U, \emptyset, left) \xrightarrow[]{(\{w_1,w_2 \},left)} (\{w_3, h_1, h_2, h_3\}, \{w_1, w_2\}, right) \xrightarrow[]{(\{w_2 \},right)} (\{w_2, w_3, h_1, h_2, h_3\}, \{w_1\}, left).
\]

Then again the permutation $\pi$ must be applied to the entire constructed part of the solution. Continuing in the same way, we obtain the following solution:
\begin{eqnarray*}
\begin{split}
&(\{w_1, w_2, w_3, h_1, h_2, h_3\}, \emptyset, left) \to (\{w_3, h_1, h_2, h_3\}, \{w_1, w_2\}, right) \to  \\
& (\{w_2, w_3, h_1, h_2, h_3\}, \{w_1\}, left) \to (\{h_1, h_2, h_3\}, \{w_1, w_2, w_3\}, right) \to \\ 
&(\{w_1, h_1, h_2, h_3\}, \{w_2, w_3\}, left) \to (\{w_1, h_1\}, \{w_2, w_3, h_2, h_3\}, right) \to \\
&(\{w_1, w_3, h_1, h_3\}, \{w_2, h_2\}, left) \to (\{w_1, w_3\}, \{w_2, h_1, h_2, h_3\}, right) \to  \\
&(\{w_1, w_2, w_3\}, \{h_1, h_2, h_3\}, left) \to (\{w_2\}, \{w_1, w_3, h_1, h_2, h_3\}, right) \to  \\
&(\{w_1, w_2\}, \{w_3, h_1, h_2, h_3\}, left) \to (\emptyset, \{w_1, w_2, w_3, h_1, h_2, h_3\}, right)
\end{split}
\end{eqnarray*}

While searching for the solution in accordance with the method indicated in the theorem, it is necessary to apply successively the permutations $e$, $\pi$, $\pi$, $e$, $\pi^{-1}$, $\pi$, $\pi$, $\pi^{-1}$, $e$, $\pi$, $\pi$ to the previously constructed part of the solution, before making the transitions from the first to the last, respectively.

Note that the set of permutations $\{e, \pi, \pi^{-1}\}$ forms a subgroup of the group $S_3$ that is isomorphic to the rotation subgroup of the symmetry group of a regular triangle.

\begin{remark}
Permutations that are applied while constructing a solution under the method described in Theorem 2 can only be of the form
\[
\begin{pmatrix}
1 & 2 & \dots & p & p+1 & p + 2 &\dots & n \\
n - p + 1 & n - p + 2 & \dots & n & 1 & 2 & \dots & n - p
\end{pmatrix},
\]
where $p = 1, 2, \dots, n$. Therefore, they always belong to a subgroup of $S_n$ that is isomorphic to the rotation subgroup of the symmetry group of a regular \textit{n}--gon.
\end{remark}

Now let us construct the whole set of solutions. We have:
\begin{eqnarray*}
\begin{split}
S_3(\{w_3, h_1, h_2, h_3\},\{w_1, w_2\}, right)= 
 \bigl\{&(\{w_3, h_1, h_2, h_3\}, \{w_1, w_2\}, right),\\ 
&{}(\{w_2, h_1, h_2, h_3\}, \{w_1, w_3\}, right), \\
&{}(\{w_1, h_1, h_2, h_3\}, \{w_2, w_3\}, right)  \bigr\}. 
\end{split}
\end{eqnarray*}

The set of moves that corresponds to $((2, 0), left)$ is the following:
\[
\bigl\{(\{w_1, w_2\}, left), (\{w_1, w_3\}, left), (\{w_2, w_3\}, left)\bigr\}.
\]

Each move transfers the initial state into one of the states belonged to the orbit of the first state. Then the construction is similar. The entire set of solutions is schematically shown in Figure \ref{fig1} at the top. At the bottom, this figure shows the solution of the \textit{MC} problem. Numbers denote states from the above solutions. Edges are marked by moves; the item indicating the location of the boat is omitted. The states where the boat is on the left bank are painted black. The above solution of the \textit{HW} problem is drawn with thicker lines. It can be seen from the figure that there are 216 solutions of the \textit{HW} problem that correspond to the solution of the \textit{MC} problem.

\begin{figure}[h]
\center{\includegraphics[width=0.65\linewidth]{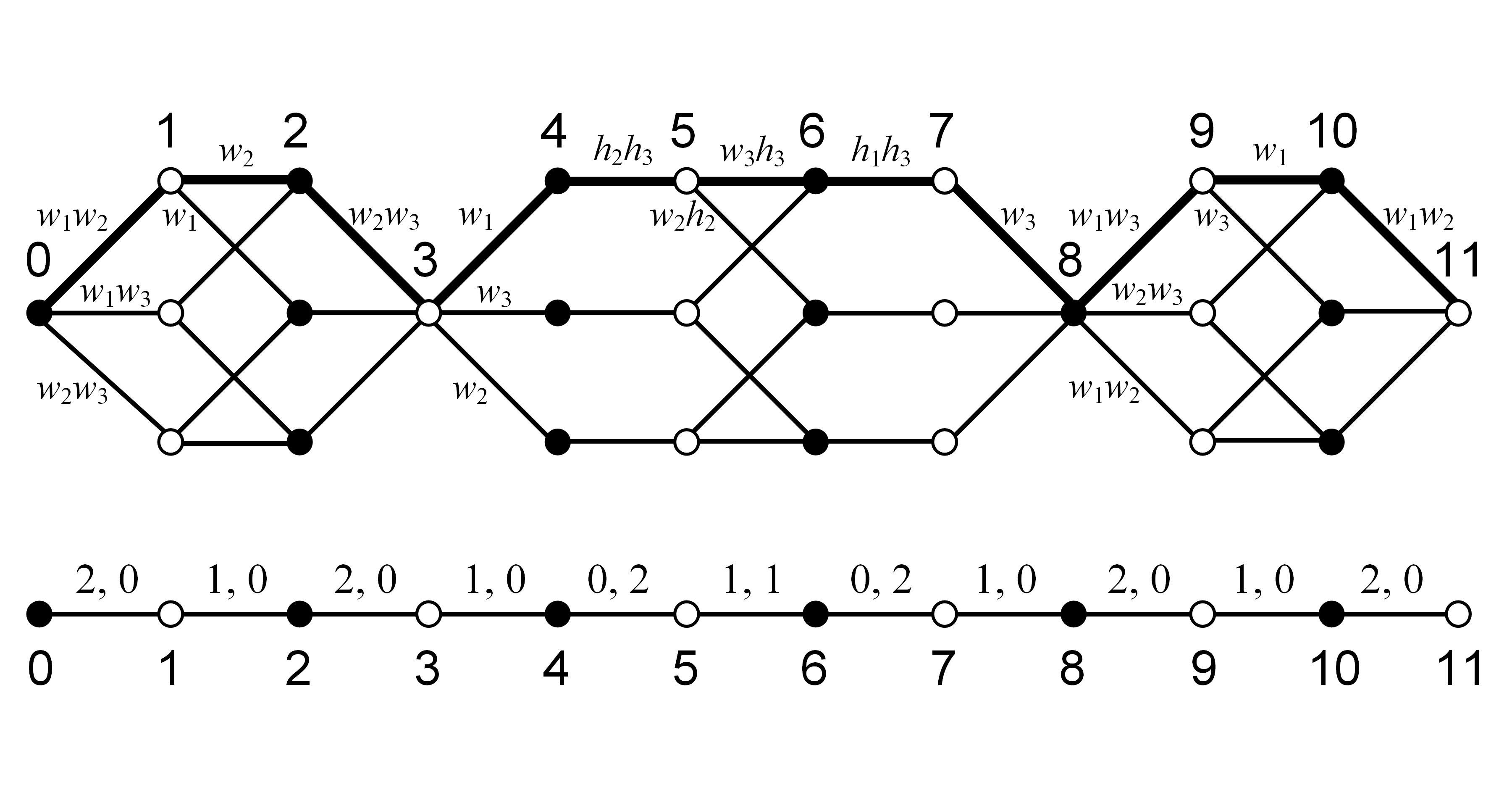}}
\caption{The set of solutions of the \textit{HW} problem and the related solution of the \textit{MC} problem, $n = 3$}
\label{fig1}
\end{figure}

It is easy to verify that the number of optimal, i.e., the shortest solutions to the problem of missionaries and cannibals for $n = 3$ is 4. The number of optimal solutions to the problem of husbands and wives for 3 couples is 486. The schemes of optimal solutions for these problems are shown in Figure \ref{fig2}.

\begin{figure}[h]
\center{\includegraphics[width=0.7\linewidth]{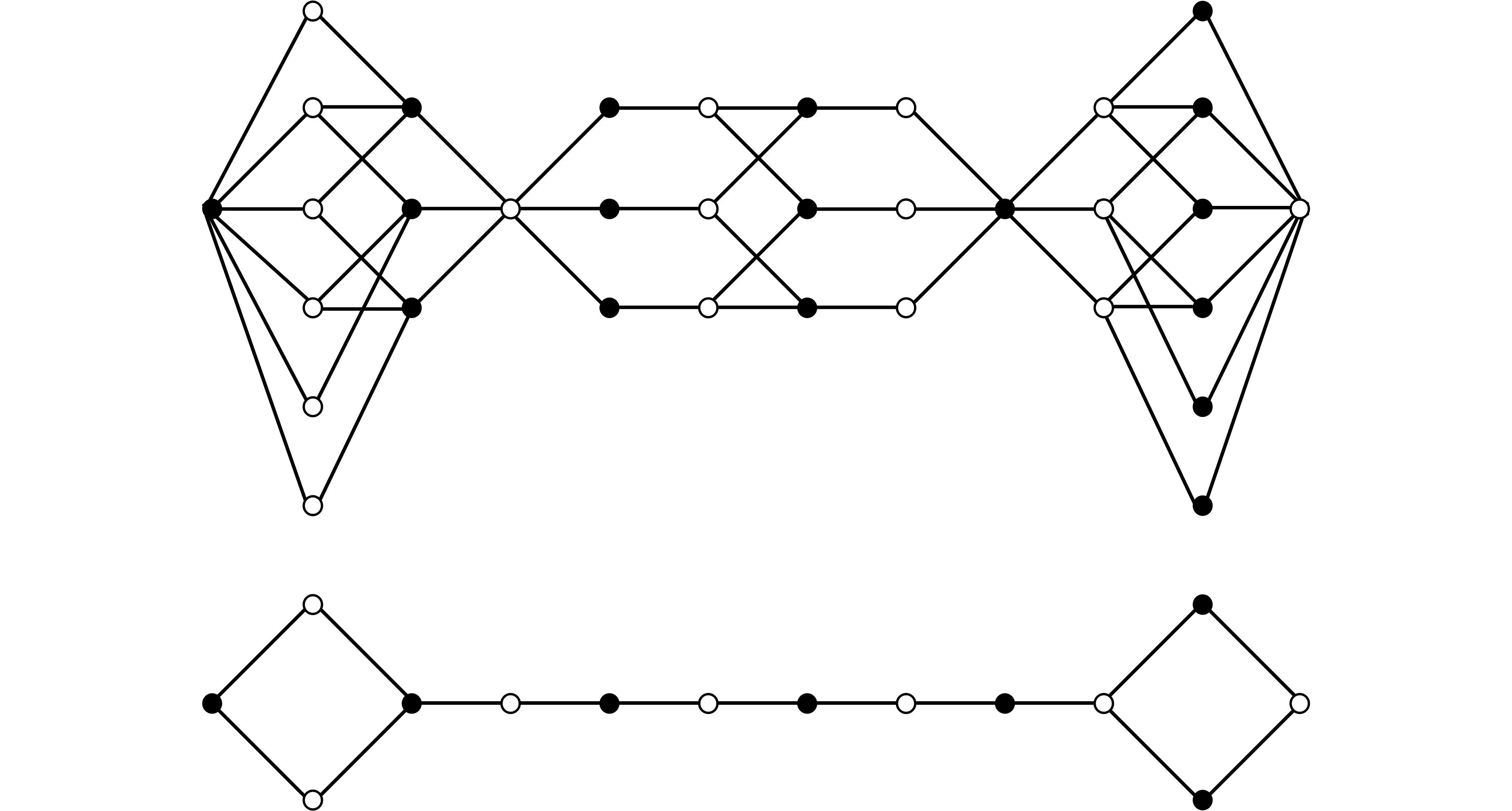}}
\caption{The schemes of optimal solutions of the \textit{HW} problem (at the top) and of the \textit{MC} problem, $n = 3$}
\label{fig2}
\end{figure}

\section{Limitation on boat capacity}
Let us show that both a two-seater boat for $n> 3$ and a three-seater boat for $n> 5$ will not be enough for the crossing.

\begin{proposition}
Suppose n is equal to 4 or 5 and $b = 2$, or $n \ge 6$ and $b = 3$. Then the \textit{MC} problem has no solution.
\end{proposition}

\begin{proof}
First suppose $n = 4$ and $b = 2$. Then there are 26 admissible states:
\begin{eqnarray*}
\begin{split}
&((p, 0), (4 - p, 4), loc) \text{ such that } p = 0, 1, 2, 3, 4;\\
&((q, 4), (4 - q, 0), loc) \text{ such that } q = 0, 1, 2, 3, 4;\\
&((r, r), (4 - r, 4 - r), loc) \text{ such that } r = 1, 2, 3,
\end{split}
\end{eqnarray*}
where $loc = left, right$.

For moves, only the cases $(1, 0)$, $(2, 0)$, $(0, 1)$, $(0, 2)$, and $(1, 1)$ are possible. The set of states is divided into 6 equivalence classes of reachable states, among which 4 classes contained isolated states, such as $((0, 0), (4, 4), left)$, $((4, 0), (0, 4), right)$, $((0, 4), (4, 0), left)$, and $((4, 4), (0, 0), right)$, and 2 classes contained 11 states each, one of which includes the initial state $((4, 4), (0, 0), left)$, and the other includes the final state $((0, 0), (4, 4), right)$; thus, the latter state is not reachable from the first. 

The picture of the state space, into which all states reachable from the initial state  and all possible transitions between them are included, is shown in Figure \ref{fig3}. Signature to vertices is an abbreviated version of our notation for states.

\begin{figure}[h]
\center{\includegraphics[width=0.5\linewidth]{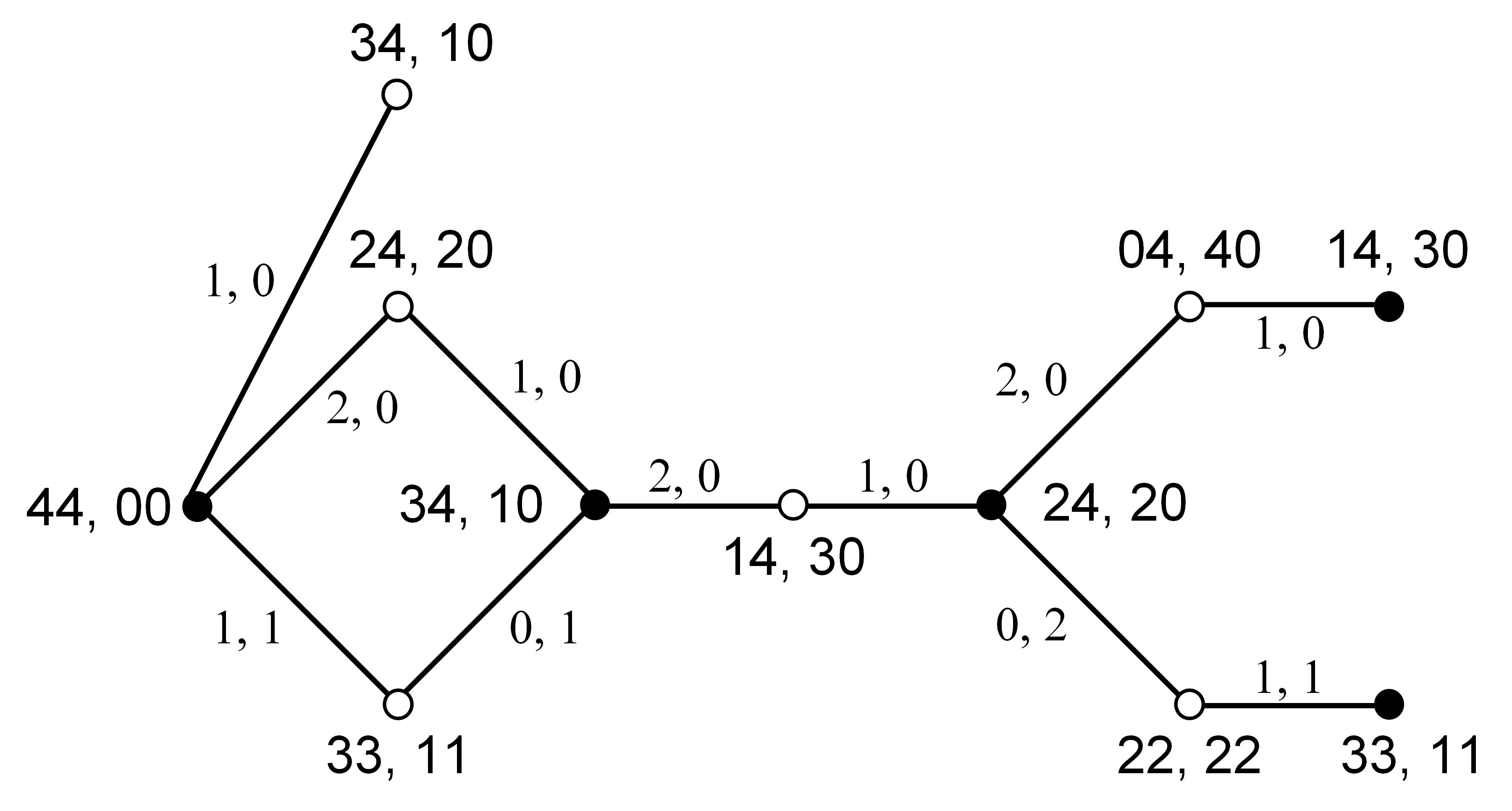}}
\caption{The state space of the \textit{MC} problem for $n = 4, b = 2$}
\label{fig3}
\end{figure}

The remaining cases are treated similarly.

In general, one can notice that if $b = 2$ and $n = 4$ or $5$ or if $b = 3$ and $n \ge 6$, then only the following states are reachable from the initial state $((n, n), (0, 0), left)$:
\begin{eqnarray*}
\begin{split}
&((n - p, n - p), (p, p), loc), p = 1, \dots, b - 1, loc = left, right;\\
&((n - b, n - b), (b, b), right);\\
&((n - q, n), (q, 0), loc), q = 1, \dots, n - 1, loc = left, right;\\
&((0, n), (n, 0), right).
\end{split}
\end{eqnarray*}

So, one can note that the connected component of the state space graph contained the initial state includes $2(n + b) - 1$ vertices in this case. 

Since the final state $((0, 0), (n, n), right)$ does not belong to the above set, it is unreachable from the initial state. So the problem has no solution. 
\end{proof}

From the proposition and the corollary of Theorem 1 it follows that the \textit{HW} problem also has no solution if $n$ is 4 or 5 and $b = 2$ or if $n > 5$ and $b = 3$. It is obvious that for $n \ge 6$ there will always be enough a four-seater boat.

\section{Category of states}
Let us turn to categories of states of the problems \textit{HW} and \textit{MC}.

First we introduce the category $Cat_X$ of the states of problem $X$, where $X$ denotes \textit{HW} or \textit{MC}, such that objects are states and morphisms are paths that connect the states on the state space graph; the identity morphism of the object is a path that consists of a single state. Let us show that $Cat_X$ is a category.

Suppose $p_1 \colon state_1 \to state_j$ and $p_2 \colon state_j \to state_k$ are morphisms of the form 
$state_1 \stackrel{f_2}\longrightarrow state_2 \stackrel{f_3}\longrightarrow \dots \stackrel{f_j}\longrightarrow state_j$ and $state_j \stackrel{f_{j+1}}\longrightarrow state_{j+1} \stackrel{f_{j+2}}\longrightarrow \dots \stackrel{f_k}\longrightarrow state_k.$
Then the morphism $p_2 \circ p_1 \colon state_1 \to state_k$ is a path that is a join of the above paths, so that it is defined as follows:
\[
state_1 \stackrel{f_2}\longrightarrow state_2 \stackrel{f_3}\longrightarrow \dots \stackrel{f_j}\longrightarrow state_j \stackrel{f_{j+1}}\longrightarrow state_{j+1} \stackrel{f_{j+2}}\longrightarrow \dots \stackrel{f_k}\longrightarrow state_k.
\]

Since the path join operation is associative, the following condition holds:
\[
(p_3 \circ p_2) \circ p_1 = p_3 \circ (p_2 \circ p_1),
\]
where $p_3 \colon state_k \to state_l$ is another morphism.

If $id_{state}$ is the identity morphism, then for morphisms $p_1 \colon state_1 \to state$ and $p_2 \colon state \to state_2$, we obviously have
\[ 
id_{state} \circ p_1 = p_1 \text{ and } p_2 \circ id_{state} = p_2.
\]

So, $Cat_X$ is a category.

Now consider the category $Cat_{HW /\!\sim}$, where objects are orbits of states of the \textit{HW} problem and morphisms are defined as follows. Suppose the morphism $p$ of the category $Cat_{HW}$ such that $p \colon state_1 \to state_k$ is a path in the state space graph of the \textit{HW} problem of the form 
\[
state_1 \stackrel{f_2}\longrightarrow state_2 \stackrel{f_3}\longrightarrow \dots \stackrel{f_k}\longrightarrow state_k.
\]
Then the sequence 
\[
S_nstate_1 \stackrel{S_nf_2}\longrightarrow S_nstate_2 \stackrel{S_nf_3}\longrightarrow \dots \stackrel{S_nf_k}\longrightarrow S_nstate_k
\]
defines a morphism in the category $Cat_{HW /\!\sim}$. 
So, there is a morphism $p'' \colon S_nstate_1 \to S_nstate_k$, for some $state_1, state_k \in V$, in the category $Cat_{HW /\!\sim}$ if there are $state_1^{(0)} \in S_nstate_1$, $\dots$, $state_k^{(0)} \in S_n state_k$ and the moves $f_2^{(0)} \in S_nf_2$, $\dots$, $f_k^{(0)} \in S_nf_k$ such that 
\[
state_1^{(0)} \stackrel{f_2^{(0)}}\longrightarrow state_2^{(0)} \stackrel{f_3^{(0)}} \longrightarrow \dots \stackrel{f_k^{(0)}}\longrightarrow state_k^{(0)}
\]
is a path in the state space graph of the \textit{HW} problem.

The identity morphism is a path that consists of a single item, as before.

Define the map $F \colon Cat_{HW} \to Cat_{HW /\!\sim}$. We put $F(state) = S_nstate$ and $F(p) = p''$, where the morphism $p'' \colon S_nstate_1 \to S_nstate_k$ corresponds to the morphism $p \colon state_1 \to state_k$ as it was described above. Suppose that $q \colon state_k \to state_l$ is another morphism in $Cat_{HW}$; then
\[
F(q \circ p) = F(q) \circ F(p).
\]

Besides, $F(id_{state})= id_{S_nstate}$, so that the following relation holds
\[F(id_{state}) = id_{F(state)}.\]

Thus, $F$ is a functor from the category $Cat_{HW}$ to the category $Cat_{HW /\!\sim}$. Likewise, one can define a functor from the category $Cat_{HW}$ to the category $Cat_{MC}$.

Theorems 1 and 2 imply the following theorem.

\begin{theorem}
The categories $Cat_{HW /\!\sim}$ and $Cat_{MC}$ are equivalent.
\end{theorem}

\begin{proof}
Consider the map $F \colon Cat_{HW /\!\sim} \to Cat_{MC}$ such that $F(S_nstate) = g(S_nstate)$ and $F(p'') = p'$, where $g \colon V/\!\sim \to V'$ is a map defined in Theorem 1, and, further, $p''$ and $p'$ are morphisms of the form $p'' \colon S_nstate_1 \to S_nstate_k$ and $p' \colon g(S_nstate_1) \to g(S_nstate_k)$ in the categories $Cat_{HW /\!\sim}$ and $Cat_{MC}$, respectively, that correspond to the path from $state_1$ to $state_k$ in the state space graph of the \textit{HW} problem, as this was described above. It is clear that $F$ is a functor.
 
Now suppose $p''\colon S_nstate_1 \to S_nstate_k$ and $q''\colon S_nstate_1 \to S_nstate_k$, where $state_1$ and $state_k$ belong to $V$, are morphisms in the category $Cat_{HW /\!\sim}$ such that $F(p'') = F(q'')$. Then it is easy to see that $p'' = q''$.  

Further, let $p' \colon state'_1 \to state'_k$ be a morphism in the category $Cat_{MC}$. Then $state'_1 = g(S_nstate_1)$ and $state'_k = g(S_nstate_k)$ for some $state_1, state_k \in V$. From Theorem 2 it follows that there exists a morphism $p''\colon S_nstate_1 \to S_nstate_k$ such that $p' = F(p'')$.

Finally, suppose $state'$ is a state in the \textit{MC} problem. Then $h(state')$, where $h$ is a map defined in Theorem 1, is an object in the category $Cat_{HW /\!\sim}$ such that $F(h(state')) = state'$. Hence, there is an isomorphism $$id_{state'}\colon F(h(state')) \to state'.$$

Thus, the functor $F$ provides an equivalence of the categories $Cat_{HW /\!\sim}$ and $Cat_{MC}$. 
\end{proof}

Now, let us define an action of the group $S_n$ on the category $Cat_{HW}$. We assume that the group acts on objects and transitions as this was defined in Section 3. Further, suppose $p \colon state_1 \to state_k$ is a morphism; then we put $\pi p \colon \pi state_1 \to \pi state_k$, so that if the morphism $p$ is a path $state_1 \stackrel{f_2}\longrightarrow state_2 \stackrel{f_3}\longrightarrow \dots \stackrel{f_k}\longrightarrow state_k$, then the morphism $\pi p$ is a path
\[
\pi state_1 \stackrel{\pi f_2}\longrightarrow \pi state_2 \stackrel{\pi f_3}\longrightarrow \dots \stackrel{\pi f_k}\longrightarrow \pi state_k.
\]

Finally, let us consider the orbit category $Orb(Cat_{HW})$. The objects of it are objects of $Cat_{HW}$, i.e., states of the \textit{HW} problem. The set of morphisms from $state_1$ to $state_2$ is a disjoint union of sets of morphisms of $Cat_{HW}$ from $state_1$ to $\pi state_2$ for each $\pi \in S_n$. The identity morphism of $state$ is the same as in the category $Cat_{HW}$. 

Consider the morphism $p_2 \circ p_1 \colon state_1 \to state_k$, where $p_1\colon state_1 \to state_j$ and $p_2\colon state_j \to state_k$ are morphisms in the category $Orb(Cat_{HW})$. According to the definition, there are permutations $\pi_1$ and $\pi_2$ such that $p_1$ and $p_2$ are morphisms $q_1 \colon state_1 \to \pi_1state_j$ and $q_2\colon state_j \to \pi_2state_k$, respectively, in the category $Cat_{HW}$, so that $p_2 \circ p_1$ is the morphism $\pi_1q_2 \circ q_1\colon state_1 \to \pi_1\pi_2 state_k$ in the category $Cat_{HW}$.

\begin{theorem}
There is a functor from the category $Orb(Cat_{HW})$ to the category $Cat_{MC}$.
\end{theorem}
\begin{proof}
Define the map $F \colon Orb(Cat_{HW}) \to Cat_{MC}$. We put $F(state) = g(S_nstate)$, so that $F$ takes the state $(L, R, loc)$ to $\left((| L |_W, | L |_H), (| R |_W, | R |_H), loc\right)$. If $p$ is a morphism of the form $p \colon state_1 \to state_k$ in the category $Orb(Cat_{HW})$,  i.e., a path from $state_1$ to $\pi state_k$ for some permutation $\pi$ in state space graph of the \textit{HW} problem, then $F(p)$ is a morphism from $g(S_nstate_1)$ to $g(S_nstate_k)$ in the category $Cat_{MC}$, i.e., a path in the state space graph of the \textit{MC} problem, according to the correspondence described in Theorem 1. 

The condition
\[F(p_2 \circ p_1) = F(p_2) \circ F(p_1)\]
holds because $S_nstate = S_n\pi state$ for $\pi \in S_n$ and $state \in V$. It is clear that the condition
\[F(id_{state}) = id_{F(state)}\]
is also satisfied, so that $F$ is a functor from the category $Orb(Cat_{HW})$ to the category $Cat_{MC}$. 
\end{proof}

\section{Conclusion}
The use of algebraic methods makes it possible to establish a natural connection between the two problems, which was considered above. Now let us turn to historical remarks. The history of the jealous husbands problem from Alcuin to Tartaglia is described in \cite{franki}. The later history, including the transition to the problem of missionaries and cannibals, is given, for example, in \cite{singmaster}.

The history of the jealous husbands problem developed in the direction of generalization and modification. There appeared variants of the problem for more couples with an understanding of the necessity to increase the capacity of the boat. For instance, in the beginning of the 16th century Luka Pacioli noticed that for four or five couples a three-seater boat is required (see \cite{franki}); Tartaglia in 1556, in fact, changed the condition of the task, allowing unsafe situations, before some people entered the boat to sail to the other side. And there appeared an island that provided the opportunity to transport any number of couples with a two-seater boat (De Fontenay, 1879, see \cite{singmaster}). The formulation of the problem changed from friends and their sisters to jealous husbands with wives, and then to masters and valets (1624) (see \cite{franki}). The latter was indicative because one of the first known variants of the missionaries and cannibals problem was about servants who robbed masters if they were more numerous than masters (1881) (see \cite{singmaster}). 
In Russia, instead of husbands and wives, knights and squires are usually used with the requirement that a squire without his knight, because of his cowardice, cannot be in the presence of other knights. It had started no later than the 1970s, when the second edition of Ignatiev's book (1979) "In the realm of savvy" \cite{ignatiev} was published, where husbands and wives from the first edition (1908) were replaced with knights and squires.

The solution of the jealous husbands problem given in the form 
\begin{center}
\textit{Women, woman, women, wife, men, man and wife,\\
Men, woman, women, man, man and wife}
\end{center}
in the 13th century (see \cite{franki}), which is translated into English in \cite{singmaster}, is actually a solution of the problem formulated as follows. \textit{
Three couples must cross a river by a two-seater boat. Nowhere ashore or in the boat the number of women should be greater than the number of men. How can they cross the river?}

But such a problem has not arisen for a long time. A possible reason is the following: earlier formulations of the \textit{HW} problem involved various kinds of discrimination existing in the society (see \cite{singmaster}), whereas the men-women formulation of the \textit{MC} problem does not have a straightforward social meaning. As it also was noticed in \cite{singmaster}, the first mentioning of the problem of missionaries and cannibals was in 1879, and by 1891 it was considered as well-known  (in addition to masters and servants, there was a variant about explorers and natives).

So, the problem of missionaries and cannibals most likely appeared in the late 1870s. Just at that time, the group approach began to spread widely. As early as in 1830, Galois obtained far-reaching results and, in particular, solved the problem of solvability of equations in radicals, using an action on the set of roots by permutations, but, as is known, his writings were not published until 1843. In 1872, Felix Klein, who later became a famous popularizer of mathematics, proclaimed a group approach in the Erlangen program. With the help of the group approach, in particular, the known problems of antiquity were solved. It became actively used not only in mathematics, but also in physics and other sciences (see, for example, \cite{ji}). Thus, on the wave of popularization of the group approach, it became possible to establish a connection between the problems described in Section 3 right at the time when the missionaries and cannibals problem  appeared. It is not known whether such a connection has really been established before now. At least, one can conclude that the problem arose in the appropriate moment of time.

\end{document}